\documentclass[ 11pt,reqno]{amsart}

\usepackage[T1]{fontenc}
\usepackage{lmodern}
\usepackage{amsmath,amssymb,mathtools}
\usepackage{dsfont}
\usepackage{microtype}
\usepackage[hidelinks]{hyperref}
\usepackage[top=1.5in, bottom=1.5in, left=1.23in, right=1.23in]{geometry}
\usepackage{orcidlink}

\numberwithin{equation}{section}

\newcommand{\C}{\mathbb C}
\newcommand{\R}{\mathbb R}
\newcommand{\Z}{\mathbb Z}
\newcommand{\T}{\mathbb T}
\newcommand{\N}{\mathbb N}
\newcommand{\M}{\mathcal M}
\newcommand{\eps}{\varepsilon}

\newtheorem{theorem}{Theorem}[section]
\newtheorem{proposition}[theorem]{Proposition}
\newtheorem{lemma}[theorem]{Lemma}
\theoremstyle{remark}

\title[Dimension-free estimates for discrete  maximal functions over cubes in $\mathbb Z^d$]
{Dimension-free estimates for discrete  maximal functions over cubes in $\mathbb Z^d$}

\author{Mariusz Mirek \orcidlink{0000-0001-7641-9893}}
\address{Department of Mathematics, Rutgers University, Piscataway, NJ 08854, USA; and Instytut Matematyczny, Uniwersytet Wroc\l awski, Plac Grunwaldzki 2, 50-384 Wroc\l aw, Poland}
\email{mariusz.mirek@rutgers.edu}

\author{Tomasz Z. Szarek \orcidlink{0000-0003-0821-5607}}
\address{Department of Mathematics, University of Georgia, Athens, GA 30602, USA; and Instytut Matematyczny, Uniwersytet Wroc\l awski, Plac Grunwaldzki 2, 50-384 Wroc\l aw, Poland}
\email{tzszarek@uga.edu}

\author{B{\l}a{\.z}ej Wr{\'o}bel \orcidlink{0000-0003-0413-8931}}
\address{Instytut Matematyczny Polskiej Akademii Nauk, ul. \'Sniadeckich 8, 00-656 Warszawa, Poland; and Instytut Matematyczny, Uniwersytet Wroc\l awski, Plac Grunwaldzki 2, 50-384 Wroc\l aw, Poland}
\email{bwrobel@impan.pl}

\thanks{Mariusz Mirek was partially supported by the NSF CAREER grant (DMS-2236493).  
	Tomasz Z.\ Szarek was supported by
	the Simons Foundation grant SFI-MPS-TSM-00013714.
	Tomasz
Z.\ Szarek and B{\l}a{\.z}ej Wr{\'o}bel were supported by the National
Science Centre, Poland, grant Sonata Bis 2022/46/E/ST1/00036. For the purpose of Open Access, the authors have applied a
CC-BY public copyright licence to any Author Accepted Manuscript (AAM) version arising from this submission.}

\begin{document}

\begin{abstract}
In this short note, we establish dimension-free $\ell^p(\mathbb Z^d)$
bounds, for all $p\in(1,\infty]$, for the discrete Hardy--Littlewood
maximal functions associated with cubes in $\mathbb Z^d$, answering a
question that had been open for a while.  The key idea is to prove
dimension-free bounds for the $\ell^p(\mathbb Z^d)$ norms of the
differences of the corresponding averages. This follows from an ad hoc
interpretation of the associated discrete multipliers as a special
continuous family of multipliers to which basic fractional integration and
complex interpolation can be applied.  The same method also yields an
elementary proof of Bourgain's dimension-free $L^p(\mathbb R^d)$
bounds for the Hardy--Littlewood maximal function associated with
cubes in $\mathbb R^d$.

\end{abstract}


\maketitle

\section{Introduction}

For $N\in\N$, let $Q_N:=[-N,N]^d\cap\Z^d$ denote the cube in $\Z^d$, and for any finitely supported function $f:\Z^d\to\mathbb C$ we define the corresponding discrete averages over cubes by
\[
\M_N f(x):=\frac1{|Q_N|}\sum_{y\in Q_N}f(x-y),
\qquad x\in\Z^d.
\]
Our main result is the following result.

\begin{theorem}\label{thm:maximal}
For every $p\in(1,\infty]$ there exists a constant $C_p>0$,
independent of the dimension $d\in\N$, such that for every
$f\in \ell^p(\Z^d)$, we have
\begin{align}
\label{eq:1}
\big\|\sup_{N\in\N}|\M_Nf|\big\|_{\ell^p(\Z^d)}\le C_p\|f\|_{\ell^p(\Z^d)}.
\end{align}
\end{theorem}

The range $p\in(1,\infty]$ for which inequality \eqref{eq:1} holds is sharp in the sense that, by combining Aldaz's weak type $(1,1)$ theorem \cite{Ald} with the result of the first and third authors with Kosz and Plewa \cite[Theorem~1]{KMPW}, one can show that the constant in the weak type $(1,1)$ inequality for the  maximal function  from \eqref{eq:1} necessarily depends on the dimension.

The dimension-free estimate for the discrete Hardy--Littlewood maximal function in Theorem~\ref{thm:maximal} was proved by the first and third authors together with Bourgain and Stein in \cite{BMSW} for $p\in(3/2,\infty]$, leaving open the dimension-free bounds \eqref{eq:1} in the remaining range $p\in(1,3/2]$. In this short note, we answer this question in the affirmative by giving an ad hoc proof, which also yields Bourgain's \cite{Bcube} dimension-free $L^p(\mathbb R^d)$ bounds for the Hardy--Littlewood maximal function associated with cubes in $\mathbb R^d$.

The latter result, on the one hand, can be deduced from \cite{KMPW}, where it was shown that the $\ell^p(\mathbb Z^d)$ norm of the discrete Hardy--Littlewood maximal operator associated with averages over any symmetric convex body (specifically over cubes) dominates the $L^p(\mathbb R^d)$ norm of its continuous analogue. On the other hand, our method can be adapted directly to the continuous setting, yielding a proof that significantly simplifies the arguments in \cite{Bcube}, whereas Bourgain's methods in \cite{Bcube} are not, at least directly, adaptable to the discrete setting.

An important outcome of \cite{BMSW} was that dimension-free estimates for the discrete maximal function from \eqref{eq:1}, as well as for other maximal functions associated with convex symmetric bodies \cite{BMSW1, Car1, Mul}, whether discrete or continuous; see \cite[inequality (47)]{BMSW4}, can be reduced to a purely operator-theoretic problem, with no need to deal directly with maximal functions: namely, to establishing dimension-free bounds for the differences of the corresponding averages.
More precisely, it suffices to prove the following estimate.

\begin{theorem}\label{thm:difference}
For every $p\in(1,\infty)$ there exists a constant $C_p>0$,
independent of the dimension $d\in\N$, such that for every
$f\in \ell^p(\Z^d)$ and for all $N\in \N$, we have
\begin{align}
\label{eq:3}
\|\M_{N+1}-\M_N\|_{\ell^p(\Z^d)\to\ell^p(\Z^d)}
\le C_p N^{-1}.
\end{align}
\end{theorem}
Theorem \ref{thm:difference} implies that for every $p\in(1,\infty)$ there exists a constant $C_p>0$,
independent of the dimension $d\in\N$, such that for every
$f\in \ell^p(\Z^d)$, we have
\begin{align}
\label{eq:4}
\sup_{R\in\N}\|\sup_{R\le N<2R}|(\M_N-\M_R)f|\|_{\ell^p(\Z^d)}\le C_p\|f\|_{\ell^p(\Z^d)}.
\end{align}
Inequality \eqref{eq:4}, in turn, combined with the dyadic maximal counterpart of \eqref{eq:1}, that is \eqref{eq:1} with $\sup_{n\in\N}|\M_{2^n}f(x)|$,  and the almost-orthogonality principle in the spirit of Carbery~\cite{Car1}, both already available in \cite{BMSW}, yields Theorem~\ref{thm:maximal}. We therefore concentrate on Theorem~\ref{thm:difference}.

The idea is very simple: we interpret the family $(\M_N)_{N\in\N}$ as part of a more general family of operators with a continuous time parameter, while preserving the product structure, which will later be crucial for deriving the appropriate dimension-free estimates. For this purpose, we observe that $\M_N$ is convolution with $\mu_N^{\otimes d}$, where
\[
\mu_N:=\frac1{2N+1}\sum_{k=-N}^N\mathds{1}_{\{k\}},\qquad N\in\N.
\]
Considering the measure $\sigma_N:=\frac12\bigl(\mathds{1}_{\{N+1\}}+\mathds{1}_{\{-(N+1)\}}\bigr)$, and 
setting $\beta_N:=\frac2{2N+3}$ so that $1-\beta_N=\frac{2N+1}{2N+3}$, we see that
\begin{equation}\label{eq:refresh}
\mu_{N+1}=(1-\beta_N)\mu_N+\beta_N\sigma_N.
\end{equation}
For $0\le s<1$, we  now define a probability measure by
\[
\lambda_{N, s}:=(1-s)\mu_N+s\sigma_N.
\]
Thus $\lambda_{N, 0}=\mu_N$, and $\lambda_{N, \beta_N}=\mu_{N+1}$. The family $(\lambda_{N, s})_{s\in[0,1)}$ provides a natural starting point for exploiting a basic form of fractional integration and complex interpolation, which was a missing ingredient in \cite{BMSW}, in order to prove Theorem~\ref{thm:difference}. We begin by establishing simple properties of $(\lambda_{N, s})_{s\in[0,1)}$. First, we gather some simple notation.
\subsection{Notation} The sets $\N = \{1,2,\ldots \}$, $\Z$, $\R$, $\C$, and $\T:=\R/\Z$ have their standard meanings. For any $p\in[1,\infty)$, we denote by $\ell^p(\Z^d)$ the space of all functions $f:\Z^d\to\C$ whose modulus has summable $p$-th power, while $\ell^\infty(\Z^d)$ denotes the space of all bounded functions on $\Z^d$. The spaces $L^p(\R)$ with $p\in[1,\infty]$, are the standard Lebesgue spaces with respect to Lebesgue measure on $\R$, and similarly for $\T$.
We set $e(\theta):=e^{2\pi i\theta}$ for every $\theta\in\R$. For every $f\in\ell^1(\Z^d)$, we define its Fourier transform by
\[
\mathcal F_{\Z^d}f(\xi):=\sum_{x\in\Z^d}e(x\cdot \xi)f(x),
\qquad \xi\in\T^d,
\]
where $x\cdot \xi:=\sum_{j=1}^d x_j\xi_j$ denotes the standard scalar product. 
Throughout the paper, we shall abbreviate $\mathcal F_{\Z^d}f$ by $\widehat f$. For two functions $f \colon X\to \mathbb C$ and $g \colon X\to (0, \infty)$, we write $f = O(g)$ if there exists a constant $C>0$ such that $|f(x)| \le C g(x)$ for all $x\in X$.

\section*{Acknowledgments} We used AI tools to test various candidates for embedding $\mu_N^{\otimes d}$ into continuous families of measures. All other ideas and intuitions, including the study of dimension-free estimates for the norms of differences of the corresponding averages and the use of fractional integration and complex interpolation, were developed by the authors.

\section{One-dimensional estimates}
Write $a_N(x):=\widehat{\mu_N}(x)$ and $b_N(x):=\widehat{\sigma_N}(x)$, which can be written explicitly
\[
a_N(x)
=\frac1{2N+1}\sum_{k=-N}^N e(kx),
\qquad
b_N(x)=\cos(2\pi(N+1)x).
\]
\begin{lemma}\label{lem:oned}
For every $N\in\N$ and $x\in\T$, we have
\begin{equation}\label{eq:lemma-seven}
|b_N(x)-a_N(x)|\le 7\bigl(1-|a_N(x)|\bigr).
\end{equation}
\end{lemma}

\begin{proof}
Since $a_N$ is real, we may write
\begin{equation}\label{eq:1-a}
1-a_N(x)=\frac2{2N+1}\sum_{k=1}^N
\bigl(1-\cos(2\pi kx)\bigr).
\end{equation}
For $1\le k\le N$, we have $1-e((N+1)x)
=(1-e(kx))
+e(kx)(1-e((N+1-k)x))$,
so
\[
|1-e((N+1)x)|^2
\le
2|1-e(kx)|^2
+2|1-e((N+1-k)x)|^2.
\]
Summing in $1\le k\le N$ gives
\[
N|1-e((N+1)x)|^2
\le4\sum_{k=1}^N|1-e(kx)|^2.
\]
Since $|1-e(\theta)|^2=(1-e(\theta))(1-e(-\theta))=2(1-\cos(2\pi \theta))$, then by \eqref{eq:1-a}, we obtain
\begin{align*}
1-b_N(x)
&\le\frac4N\sum_{k=1}^N
\bigl(1-\cos(2\pi kx)\bigr)\notag\\
&=\frac{2(2N+1)}N(1-a_N(x))\notag\le6(1-a_N(x)).
\end{align*}
If $a_N(x)\ge0$, then by the previous inequality, we obtain
\begin{align*}
|b_N(x)-a_N(x)|
\le(1-a_N(x))+(1-b_N(x))
\le7(1-a_N(x))
=7(1-|a_N(x)|).
\end{align*}

Suppose now that $a_N(x)<0$. By periodicity and symmetry of $a_N(x)$ and $b_N(x)$,
it is enough to consider $0\le x\le\tfrac12$.
On $(0,\tfrac12]$ the denominator in
\[
a_N(x)=\frac{\sin((2N+1)\pi x)}{(2N+1)\sin(\pi x)}
\]
is positive. Hence $a_N(x)<0$ implies that $x$ lies beyond
the first positive zero of the numerator, and therefore satisfies $\frac{1}{2} \ge x>\frac1{2N+1}.$
Using $\sin(\pi x)\ge2x$ on $[0,\tfrac12]$, we obtain
\[
(2N+1)\sin(\pi x)>2,
\]
and hence $a_N(x)\ge-\frac12$.
Therefore, $-|a_N(x)|\ge-\frac12$, and consequently
\[
|b_N(x)-a_N(x)|
\le2
\le4(1-|a_N(x)|)
\le7(1-|a_N(x)|).
\]
This proves \eqref{eq:lemma-seven}, and the proof of the lemma is completed.
\end{proof}

\section{Dimension-free multiplier estimates}
Set $z_{N,s}(x):=(1-s)a_N(x)+sb_N(x)$ for $s\in [0,1)$ and note that
\[
m_{N,s}(\xi):=\widehat{\lambda_{N, s}^{\otimes d}}(\xi)=\prod_{j=1}^d z_{N,s}(\xi_j),
\qquad \xi\in\T^d.
\]
Fix $s_0\in(0,1)$ and assume that $0\le s\le s_0$. For $1\le j\le d$, define $u_j:=1-|a_N(\xi_j)| \in [0,1]$, then
by Lemma~\ref{lem:oned}, we may write
\begin{equation}\label{eq:diffab}
|b_N(\xi_j)-a_N(\xi_j)|\le7u_j.
\end{equation}
Moreover, $|z_{N,s}(\xi_j)|\le(1-s)|a_N(\xi_j)|+s|b_N(\xi_j)|$, and consequently
\begin{align}
\label{eq:zdecay}
|z_{N,s}(\xi_j)|
\le(1-s)(1-u_j)+s
=1-(1-s)u_j
\le e^{-(1-s_0)u_j},
\end{align}
since $-(1-s)\le -(1-s_0)$.

\begin{proposition}\label{prop:derivative}
For every $k\in\N \cup \{0 \}$ and $s_0\in(0,1)$, there exists $C_{k,s_0}>0$ independent of $d\in\N$ such that
\[
\sup_{N\in\N}\sup_{d\in\N}\sup_{0\le s\le s_0}\sup_{\xi\in\T^d}
|\partial_s^k m_{N,s}(\xi)|\le C_{k,s_0}.
\]
\end{proposition}

\begin{proof}
The conclusion is trivial for $k=0$, in which case we may take $C_{0,s_0}=1$. Therefore, we may assume that $k\in\N$.
Since $z_{N,s}$ is affine in $s$, we obtain
\[
\partial_s^k m_{N,s}(\xi)
=
k!\sum_{\substack{E\subseteq\{1,\dots,d\}\\|E|=k}}
\prod_{j\in E}\bigl(b_N(\xi_j)-a_N(\xi_j)\bigr)
\prod_{j\notin E}z_{N,s}(\xi_j).
\]
Let $U:=\sum_{j=1}^d u_j$, and $c:=1-s_0>0$.
By \eqref{eq:diffab} and \eqref{eq:zdecay}, we have 
\begin{align*}
|\partial_s^k m_{N,s}(\xi)|
&\le
k!7^k\sum_{|E|=k}
\Bigl(\prod_{j\in E}u_j\Bigr)
\exp\Bigl(-c\sum_{j\notin E}u_j\Bigr).
\end{align*}
Since $0\le u_j\le1$ and $|E|=k$, we may further write
\begin{align*}
\exp\Bigl(-c\sum_{j\notin E}u_j\Bigr)
=e^{-cU}\exp\Bigl(c\sum_{j\in E}u_j\Bigr)
\le e^{ck}e^{-cU}.
\end{align*}
It follows that
\begin{equation}\label{eq:elementary-sym}
|\partial_s^k m_{N,s}(\xi)|
\le
k!7^ke^{ck}e^{-cU}
\sum_{|E|=k}\prod_{j\in E}u_j.
\end{equation}
Expanding $U^k=(\sum_{j=1}^d u_j)^k$, each monomial $\prod_{j\in E}u_j$ with $|E|=k$ occurs exactly $k!$ times among the terms with pairwise distinct indices. Since all remaining terms are nonnegative, we obtain
\[
\sum_{|E|=k}\prod_{j\in E}u_j\le\frac{U^k}{k!}.
\]
Inserting this bound into \eqref{eq:elementary-sym}, we conclude that there exists $C_{k,s_0}>0$ such that
\[
|\partial_s^k m_{N,s}(\xi)|
\le7^ke^{ck}e^{-cU}U^k
\le C_{k,s_0},
\]
since $\sup_{U\ge0}e^{-cU}U^k \le c^{-k} k!$. This completes the proof.
\end{proof}

\section{An interpolation lemma}
In this section we prove Proposition \ref{prop:interpolation}, which is a modification of the interpolation argument from \cite[Section~6]{MSW} and will allow us to deduce Theorem \ref{thm:difference} very quickly in the next section. The key point is that, by choosing the order of differentiation appropriately, the interpolation point can be taken to be the integer $1$.

Let $T_t f=K_t*f$, for $t\ge1$, be a family of convolution operators on $\ell^p(\Z^d)$, where $K_t\in\ell^1(\Z^d)$, and let $m_t = \mathcal F_{\Z^d} K_{t}$ denote the corresponding Fourier multipliers, that is, $\mathcal F_{\Z^d}T_t f(\xi)= m_t(\xi)\mathcal F_{\Z^d}f(\xi)$.
Let $n \in \N$ and suppose that $t\mapsto m_t(\xi)$ is of class $C^n([1,\infty))$ for every $\xi\in\T^d$ and that
\begin{equation}\label{eq:L1}
\sup_{t\ge1}\|T_t\|_{\ell^1(\Z^d)\to\ell^1(\Z^d)}\le A_0,
\end{equation}
and
\begin{equation}\label{eq:deriv-ass}
\max_{0\le j\le n}\sup_{t\ge1}\sup_{\xi\in\T^d}
 |t^j\partial_t^j m_t(\xi)|\le A_n.
\end{equation}
Here and below, $\partial_tT_t$ denotes the multiplier operator with multiplier $\partial_tm_t$.

\begin{proposition}\label{prop:interpolation}
Let $1<p<2$. Then there exists a constant $C_p>0$ such that for any integer $n\in\N$ satisfying ${p'}/2<n<p'$, we have
\[
\sup_{t\ge1}\|t\partial_tT_t\|_{\ell^p(\Z^d)\to\ell^p(\Z^d)}
\le C_p (A_0+A_n).
\]
\end{proposition}

\begin{proof}
Fix $\xi \in \T^{d}$ and consider $h_\xi(s):={m_s(\xi)}s^{-1}$. Since $\partial_s^{m}(s^{-1})=(-1)^mm!s^{-m-1}$, then 
by Leibniz' rule we have
\[
\partial_s^rh_\xi(s)
=
\sum_{j=0}^r\binom rj
\partial_s^jm_s(\xi)\,(-1)^{r-j}(r-j)!s^{-r+j-1},
\]
and consequently by \eqref{eq:deriv-ass}, for every $0\le r\le n$, we obtain
\begin{equation}\label{eq:hder-all}
|\partial_s^rh_\xi(s)|\le er! A_n s^{-r-1}.
\end{equation}
Fix $u\ge1$. For $z\in S_n:=\{z\in\C:-1<\operatorname{Re} z<n\}$, we define
\begin{equation}\label{eq:puz}
p_u^z(\xi)
:=
\frac{(-1)^n u^{z+1}}{\Gamma(n-z)}
\int_u^\infty (s-u)^{n-z-1}\partial_s^nh_\xi(s)\,ds.
\end{equation}
The function $S_n\ni z\mapsto p_u^z(\xi)$ is holomorphic by Morera's theorem and Fubini's theorem, since by \eqref{eq:hder-all} the integrand is integrable. Indeed, near $s=u$ this follows from $\operatorname{Re} z<n$, whereas at infinity the integrand is $O(s^{-\operatorname{Re} z-2})$. Thus, if $P_u^z$ denotes the multiplier operator associated with $p_u^z$, then for any $f,g\in \ell^2(\Z^d)$, Plancherel's theorem and Fubini's theorem imply that the function
$S_n\ni z\longmapsto \langle P_u^zf,g\rangle_{\ell^2(\Z^d)}$ is analytic in $S_n$. Recall that the standard vertical-line estimates for the Gamma function assert that
if $\operatorname{Re} z\ge -1$, then
\begin{align}
\label{eq:5}
\frac{1}{|\Gamma(z)|}\le 2 \Big(\sqrt{1+|\operatorname{Im} z|^2}\Big)^{1/2-\operatorname{Re} z}e^{\pi|\operatorname{Im} z|/2}.
\end{align}

We will use \eqref{eq:5} to prepare the ground for an application of Stein's complex interpolation theorem to the operators $P_u^z$; see, for example, \cite[Theorem~1.3.7 and Exercise~1.3.4]{Gra1}. For this purpose, we set $\eps={2n}/{p'}-1$ and note that $\eps\in(0,1)$, since  $p'/2<n<p'$. We will then estimate the norms of the operators $P_u^z$ on the two lines $z=n-\eps+i\tau$ and $z=-\eps+i\tau$. In Stein's complex interpolation theorem, any growth of the operator norms smaller than a constant multiple of the quantity $\exp(\exp(\gamma |\operatorname{Im}z|/n))$, with $0\le\gamma<\pi$, is admissible.

We begin by noting  that there exists a constant $C_{n, p}>0$ such that for any $f, g\in\ell^2(\Z^d)$ and  any $z\in\{z\in\mathbb C: -\eps< \operatorname{Re} z< n-\eps\}$, we have
\begin{align}
\label{eq:6}
|\langle P_u^zf,g\rangle_{\ell^2(\Z^d)}|\le C_{n, p}A_n\|f\|_{\ell^2(\Z^d)}\|g\|_{\ell^2(\Z^d)}e^{3\pi|\operatorname{Im}z|/4}.
\end{align}
Estimate \eqref{eq:6} follows from Plancherel's theorem, together with \eqref{eq:hder-all} and \eqref{eq:5}, the latter being used to control $|\Gamma(n-z)|^{-1}$ whenever $\operatorname{Re}(n-z)>\eps$, and then to estimate the multiplier in \eqref{eq:puz}, yielding the upper bound in \eqref{eq:6}. Now we examine the behavior on the boundary.
\smallskip

\paragraph{\textbf{Behavior $P_u^z$ on the line $z=n-\eps+i\tau$}} Using \eqref{eq:hder-all} and \eqref{eq:5} and the substitution $s=u(1+v)$, we deduce that there exists a constant $C_{n, p}>0$ such that
\begin{align*}
|p_u^z(\xi)|
&\le en!A_n
\frac{u^{n-\eps+1}}{|\Gamma(\eps-i\tau)|}
\int_u^\infty (s-u)^{\eps-1}s^{-n-1}\,ds\\
&=
\frac{en!A_n}{|\Gamma(\eps-i\tau)|}
\int_0^\infty v^{\eps-1}(1+v)^{-n-1}\,dv
\le C_{n, p}A_ne^{3\pi|\tau|/4}.
\end{align*}
Consequently, Plancherel theorem gives
\begin{equation}\label{eq:L2bd}
\|P_u^{n-\eps+i\tau}\|_{\ell^2(\Z^d)\to\ell^2(\Z^d)}
\le C_{n, p}A_ne^{3\pi|\tau|/4}.
\end{equation}

\paragraph{\textbf{Behavior $P_u^z$ on the line $z=-\eps+i\tau$}} 
Since $-1<\operatorname{Re} z<0$, integrating \eqref{eq:puz} by parts $n$ times gives
\begin{equation}\label{eq:leftrep}
p_u^z(\xi)
=
-\frac{zu^{z+1}}{\Gamma(1-z)}
\int_u^\infty (s-u)^{-z-1}\frac{m_s(\xi)}s\,ds.
\end{equation}
Indeed, fixing $1\le k\le n$, after the $k$-th integration by parts in \eqref{eq:puz}, the relevant boundary expression is a constant multiple of
\[
(s-u)^{n-z-k}\partial_s^{n-k}h_\xi(s).
\]
By \eqref{eq:hder-all}, this is $O(s^{-\operatorname{Re} z-1})$ as $s\to\infty$, and hence tends to zero because $\operatorname{Re} z>-1$. At $s=u$ it also vanishes, since $\operatorname{Re}(n-z-k)\ge-\operatorname{Re} z>0$. Therefore, all boundary terms vanish, and \eqref{eq:leftrep} follows by noting that
\[
\frac{(n-z-1)(n-z-2)\cdots(1-z)(-z)}{\Gamma(n-z)}=\frac{-z}{\Gamma(1-z)}.
\]

Now on the line $z=-\eps+i\tau$, formula \eqref{eq:leftrep}, \eqref{eq:L1} and \eqref{eq:5} yield that there exists a constant $C_{n, p}>0$ such that
\begin{align*}
\|P_u^z\|_{\ell^1(\Z^d)\to\ell^1(\Z^d)}
&\le
\frac{|z|u^{1-\eps}}{|\Gamma(1-z)|}
\int_u^\infty (s-u)^{\eps-1}
\|T_s\|_{\ell^1(\Z^d)\to\ell^1(\Z^d)}\,\frac{ds}{s}\\
&\le
\frac{|z|A_0}{|\Gamma(1-z)|}
\int_0^\infty \frac{v^{\eps-1}}{1+v}\,dv\le C_{n, p}A_0e^{3\pi|\tau|/4}.
\end{align*}
Thus
\begin{equation}\label{eq:L1bd}
\|P_u^{-\eps+i\tau}\|_{\ell^1(\Z^d)\to\ell^1(\Z^d)}
\le C_{n, p}A_0e^{3\pi|\tau|/4}.
\end{equation}
Since the implied constants in \eqref{eq:6}, \eqref{eq:L2bd}, and \eqref{eq:L1bd} may depend on the parameters $n$ and $p$, the exponential bounds obtained there can be dominated by a constant multiple of the double exponential bound $\exp(\exp(\gamma |\operatorname{Im}z|/n))$, for some $0\le\gamma<\pi$, which is admissible in Stein's complex interpolation theorem, which we now invoke.
Let $\theta=2/{p'}$, 
then $1/p=(1-\theta)+\theta/2$, 
and, by the definition of $\eps={2n}/{p'}-1$, we have 
\[
(1-\theta)(-\eps)+\theta(n-\eps)
=-\eps+\frac{2n}{p'}=1.
\]
Interpolating between \eqref{eq:L2bd} and \eqref{eq:L1bd} using Stein's complex interpolation theorem, we deduce that there exists a constant $C_{n,p}>0$ such that
\begin{equation}\label{eq:Pu1}
\|P_u^{(1-\theta)(-\eps)+\theta(n-\eps)}\|_{\ell^p(\Z^d)\to\ell^p(\Z^d)}
\le C_{n, p} A_0^{1-\theta}A_n^\theta
\le C_{n, p} (A_0+A_n).
\end{equation}
It remains to identify $P_u^1$.  Putting $z=1$ in \eqref{eq:puz}, we see
\[
p_u^1(\xi)
=
\frac{(-1)^nu^2}{\Gamma(n-1)}
\int_u^\infty (s-u)^{n-2}\partial_s^nh_\xi(s)\,ds.
\]
Since $n>p'/2>1$, we have $n\ge2$, then repeated integration by parts gives
\begin{align*}
\int_u^\infty (s-u)^{n-2}\partial_s^nh_\xi(s)\,ds
&=
(-1)^{n-2}(n-2)!\int_u^\infty h_\xi''(s)\,ds\\
&=
(-1)^{n-1}(n-2)!h_\xi'(u).
\end{align*}
Here the intermediate boundary terms at $s=u$ vanish because they contain a positive power of $s-u$, while the boundary terms at infinity are $O(s^{-2})$ by \eqref{eq:hder-all}. In the last equality we also used $h_\xi'(s)\to0$ as $s\to\infty$. Since $\Gamma(n-1)=(n-2)!$, it follows that
\begin{align*}
p_u^1(\xi)
=-u^2h_\xi'(u)
=m_u(\xi)-u\partial_um_u(\xi).
\end{align*}
Thus $P_u^1=T_u-u\partial_uT_u$. By \eqref{eq:L1} and Young's convolution inequality, $\|T_u\|_{\ell^p(\Z^d)\to\ell^p(\Z^d)}\le A_0$, and consequently, together with \eqref{eq:Pu1}, this proves the proposition.
\end{proof}

\section{Proof of the difference estimate from Theorem \ref{thm:difference}}

Choose $s_0\in(2/5,1)$ and a smooth function $\rho:[1,\infty)\to[0,s_0]$
such that $\rho(t):=t-1$ for $1\le t\le1+\frac25$,
and $\rho$ is constant for all sufficiently large $t\ge 1$. Let $T_tf:=\lambda_{N, \rho(t)}^{\otimes d}*f$, since $\lambda_{N, \rho(t)}^{\otimes d}$ is a probability measure, then $\|T_t\|_{\ell^1(\Z^d)\to\ell^1(\Z^d)}\le1$.
By Proposition~\ref{prop:derivative} and the chain rule, for every $k\in\N \cup\{0\}$, there exists a constant $D_k>0$ such that
\[
\sup_{t\ge1}\sup_{\xi\in\T^d}
 |t^k\partial_t^km_{N, \rho(t)}(\xi)|\le D_k,
\]
uniformly in $N$ and $d$, since, all derivatives of $\rho$ are supported in a fixed compact subset of $[1,\infty)$.
Proposition~\ref{prop:interpolation} therefore gives, for every $1<p<2$, a constant $C_p>0$ such that
\[
\|t\partial_tT_t\|_{\ell^p(\Z^d)\to\ell^p(\Z^d)}\le C_p.
\]
Since $N\ge1$, we have $\beta_N\le2/5$, and hence $T_1=\M_N$ and $T_{1+\beta_N}=\M_{N+1}$.
Therefore
\begin{align*}
\|\M_{N+1}-\M_N\|_{\ell^p(\Z^d)\to\ell^p(\Z^d)}
&\le
\int_1^{1+\beta_N}\|\partial_tT_t\|_{\ell^p(\Z^d)\to\ell^p(\Z^d)}\,dt\\
&\le C_p
\int_1^{1+\beta_N}\frac{dt}{t}\\
&\le C_p\beta_N.
\end{align*}
This proves Theorem~\ref{thm:difference} for $1<p<2$, since $\beta_N\le N^{-1}$. For $p=2$, the same conclusion follows directly from Proposition~\ref{prop:derivative} and the Plancherel theorem. Finally, since $\M_{N+1}-\M_N$ is self-adjoint, the range $p>2$ follows by duality.

\end{document}